\documentclass[letterpaper, 10 pt]{IEEEtran}  
\IEEEoverridecommandlockouts                              

\usepackage{graphics,subfigure} 
\usepackage{epsfig} 
\usepackage{mathptmx} 
\usepackage{times} 
\usepackage{amsmath} 
\usepackage{amssymb}  
\usepackage{amsthm}

\usepackage[noend]{algpseudocode}
\usepackage{siunitx}
\usepackage{longtable,tabularx}
\usepackage[T1]{fontenc}
\usepackage[utf8]{inputenc}
\usepackage{graphicx}
\usepackage{subfigmat, pgfplots, bm}
\usepackage{subfigmat}

\usepackage{multirow,multicol,threeparttable}
\usepackage{booktabs}
\usepackage{enumerate}

\usepackage{float}
\usepackage{setspace}

\usepackage{algorithm}
\usepackage{algorithmicx}
\algnewcommand\algorithmicinput{\textbf{Input:}}
\algnewcommand\Input{\item[\algorithmicinput]}

\algnewcommand\algorithmicoutput{\textbf{Output:}}
\algnewcommand\Output{\item[\algorithmicoutput]}

\title{\LARGE \bf
Optimal Target Intercept Paths for Vehicles with Turn Radius Constraints
}

\author{Satyanarayana Gupta Manyam$^{1}$ \and David Casbeer$^{2}$ %
\and  Alexander Von Moll$^{3}$ \and Zachariah Fuchs$^{4}$
\thanks{$^{1}$Infoscitex Corp.,
        4027 Col Glenn Hwy, Dayton, OH, 45431, USA
        {\tt\small msngupta@gmail.com}}%
\thanks{$^{2}$Control Center of Excellence, Air Force Research Laboratories, 
        WPAFB, OH 45431, USA
        {\tt\small david.casbeer@us.af.mil}}%
\thanks{$^{3}$Control Center of Excellence, Air Force Research Laboratories, 
        WPAFB, OH 45431, USA
        {\tt\small alexander.von\_moll@us.af.mil}}%
\thanks{$^{4}$Department of Electrical Engineering, University of Cincinnati, Cincinnati, OH, USA
        {\tt\small fuchsze@ucmail.uc.edu}}%
}

\newtheorem{prop}{Proposition}
\newtheorem{theorem}{Theorem}
\newtheorem{lemma}{Lemma}

\begin{document}
\maketitle
\thispagestyle{empty}
\pagestyle{empty}

\begin{abstract}
We present a path planning problem for a pursuer to intercept a target traveling on a circle. The pursuer considered here has limited yaw rate, and therefore its path should satisfy the kinematic constraints. We assume that the distance between initial position of the pursuer and any point on the target circle is greater than four times the minimum turn radius of the pursuer. We prove the continuity of the Dubins paths of type Circle-Straight line-Circle with respect to the position on the target circle. This is used to prove that the optimal interception path is a Dubins path, and we present an iterative algorithm to find the optimal interception point on the target circle.

\end{abstract}
\section{INTRODUCTION} \label{sec:intro}

The use of autonomous air vehicles is ever growing in military and civil domains, and consequently, there is an increasing need for path planning. We consider an application where an autonomous air vehicle is used to pursue and intercept a target. In patrolling or surveillance scenarios involving unmanned air-vehicles, the air-vehicle may need to pursue a target, and here, we consider a scenario where the target is traveling on a circle at constant speed. In the context of pursuit-evasion, the pursuer will most likely not know the evader's future path (We refer to evader as target henceforth). Hence, the pursuer could estimate the path of the target and use a strategy to intercept the target along the target's expected path. It is reasonable to assume that the extrapolated path of the target could be a circular arc, and the pursuer must determine the optimal path to intercept at this instance. We focus on this problem of path planning for an air vehicle that needs to intercept a target traveling on a circle. 

An interception is considered to occur when the pursuer reaches some distance behind the target's tail, and its heading is aligned with that of the target. As such, we define an imaginary point, referred to as rabbit, which is located on the target circle behind target. The objective of the pursuer is to reach the rabbit, with its heading tangential to the circle. Without loss of generality, we interchangeably use rabbit/target position as the goal location for the pursuer, since the path planning methodology to arrive at these points would be the same. Also we assume that the pursuer and the target are at same altitude, and we solve the path planning in a $2$D-plane. This is reasonable for planning purpose, as a pursuer could always adjust its altitude while implementing the pursuing strategy. 

The air vehicles considered are fixed wing aircraft and therefore have yaw-rate constraints. We use the Dubins \cite{Dubins1957, bui1994shortest} model for the vehicle, which generates the paths satisfying the turn radius constraints. Here, we look at the properties of the Dubins paths starting from a given position and heading, and that ends on the target circle with the final heading being tangent to the circle. We also assume that the distance between the starting location of the pursuer and any point on the target circle is at least four times the minimum turning radius of the pursuer. Therefore one needs to consider only the Dubins paths of type Circular Arc-Straight Line-Circular Arc (CSC). The characterization of the CSC Dubins paths is done in a companion paper \cite{Manyam2018Shortest}, which we use here to prove the properties of the optimal intercepting path.

\subsection{Literature}
Dubins paths are typically used to find shortest paths satisfying minimum turn radius constraints \cite{bui1994accessibility, bui1994shortest, yang2002optimal, wong2004uav, manyam2017tightly, manyam2018tightly}. Existing results address path planning with turn radius constraints when an initial and final positions and headings are given \cite{bui1994shortest}, or an interval of headings are given \cite{manyam2017tightly, manyam2018tightly}. Some other generalizations include finding a path that passes through a given third point \cite{yang2002optimal, wong2004uav, sadeghi2016dubins}, and path planning in the presence of wind \cite{mcgee2005optimal,techy2009minimum}, or obstacles \cite{boissonnat1996polynomial, macharet2009generation, agarwal1995motion, maini2016path}. However, the path planning problem when the final position is not stationary has never been addressed, and needs attention due to its applications such as rendezvous planning, target interception problem, etc. The existing results in the literature that are related to target intercept problem presented here are trajectory planning for coordinated rendezvous in \cite{mclain2000traj, Mclain2001Coop,Beard2002Coord}, where the paths for multiple UAVs are planned to arrive at different stationary targets simultaneously. The approach is to start with a feasible path, break this path into chain links, and add or remove chain links to find paths of simultaneous arrival. The problem we present in this paper involves finding an optimal path in the presence of moving target, and the time of travel for the pursuer and the target needs to be equal, and this makes it different from the existing results. 

This paper is organized as the following: We present the description of the problem, and briefly state the solution approach in the Section \ref{sec:prb}. The properties of the Dubins paths are presented, and their continuity is proved in the Section \ref{sec:dubpaths}. The algorithm to find the optimal interception point on target circle is presented in the Section \ref{sec:optpath}, and  conclusions are made in the Section \ref{sec:concl}.
\section{Problem Description} \label{sec:prb}

Without loss of generality, we assume that the initial position of the pursuer is $(0,0)$ and the heading is $0$ degrees with respect to the $x-$axis. We know the center and radius of the target circle, and the initial position of the target on that circle. Let $C_t = (c_x,c_y)$ be the center of the target circle, and $r_t$ be the radius of the target circle. We refer to a point on the target circle by its angular position $\alpha$, which is the angle measured counter-clockwise from the $x$-axis (see Fig. \ref{fig:probdesc}). Therefore, $\alpha=0$ represents a point on the circle, such that a line connecting $C_t$ and this point is parallel with the $x$-axis. We denote the initial position of the target on the circle by $\alpha_i$. Let $\rho$ be the minimum turn radius of the pursuer, and $v_p$ and $v_t$ be the constant speeds of the pursuer and the target respectively. 

With these definitions, we now state the problem to be solved herein. Given the initial position ($p_i$) and heading ($\theta_i$) of the pursuer, initial position of the target ($\alpha_i$), their speeds ($v_p$, $v_t$), the center ($C_t$) and radius ($r_t$) of the target circle, and the minimum turn radius of the pursuer ($\rho$), find a path that starts from the initial position of the pursuer, and ends on the target circle at a point (referred to as interception point) with heading tangential to the circle, such that
\begin{itemize}
\item the curvature of the path satisfies the minimum turn radius $\rho$ of the pursuer,
\item the pursuer and target travel times are equal, and
\item the path is of minimum length.
\end{itemize}
We refer to this problem as the Intercepting Target on a Circle Problem (ITOCP).

\begin{figure}[htpb]
\centering
\includegraphics[width=0.9\columnwidth]{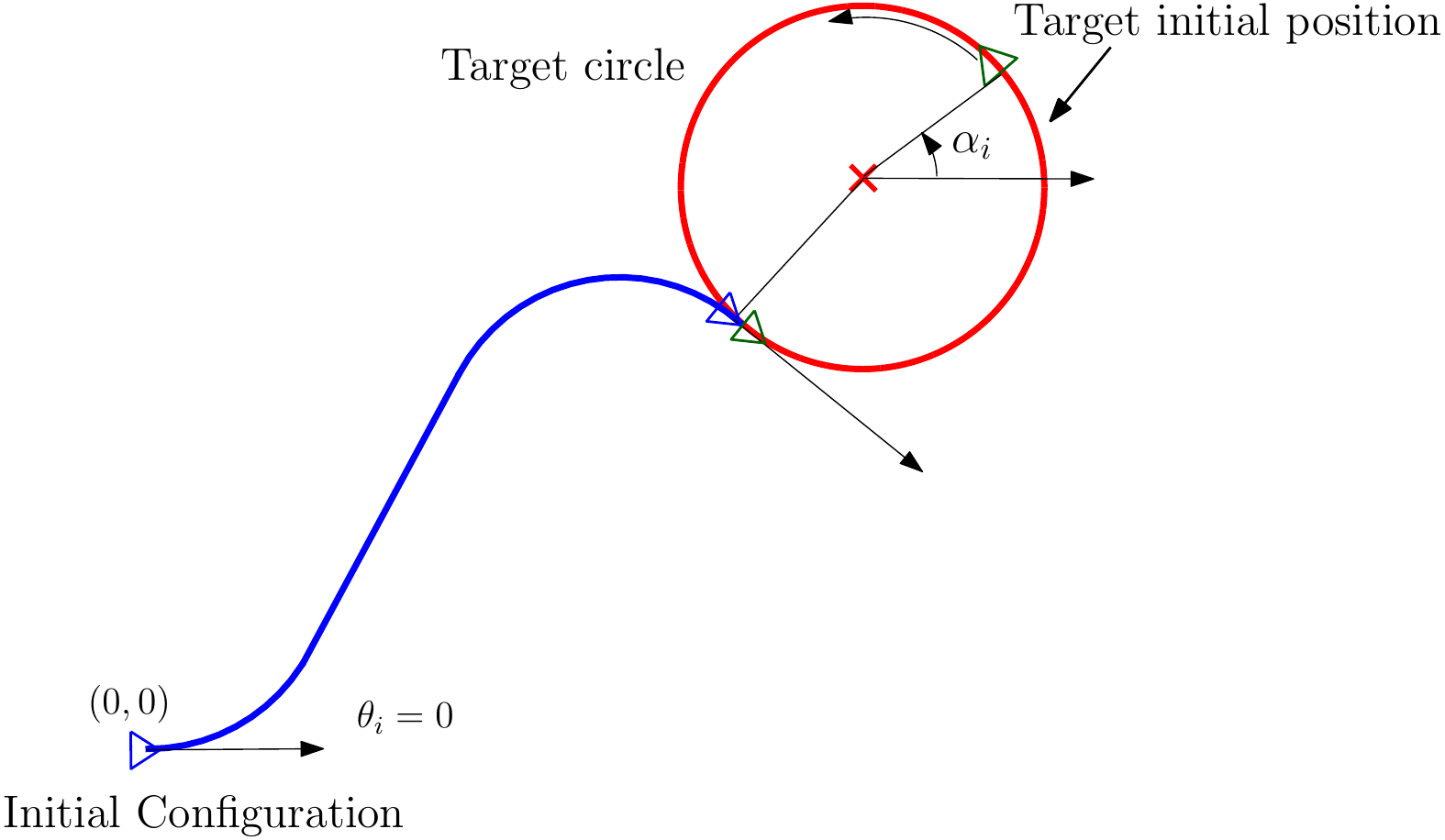}
\caption{Feasible pursuer path to intercept a target traveling on a circle.}
\label{fig:probdesc} 
\end{figure}

A feasible solution to the ITOCP is a path that starts at the pursuer's initial configuration, and ends on the target circle as shown in Fig.~\ref{fig:probdesc}. The solution to this problem is given by the optimal interception point ($\alpha^*$) and the curvature constrained path for the pursuer from its initial position to the interception point.

For given initial and final positions and headings, the well known results by Dubins \cite{Dubins1957} states that the optimal path contains three segments, where each segment could be a straight line (represented by letter S) or a circular arc of minimum turn radius (represented by letter C). The circular arcs could be clockwise or counter-clockwise (represented by R or L respectively). Dubins result says that the minimum of four different CSC paths (\textit{viz.} LSL, LSR, RSR, RSL) and two CCC paths (\textit{viz.} LRL and RLR) is the path of minimum length satisfying the curvature constraints. Furthermore, assuming the distance between the initial and final points is greater than $4\rho$, the optimal path is only of the four CSC paths~\cite{bui1994shortest, goaoc2013bounded}.

The path planning problem for ITOCP has additional temporal constraints, as the travel time for the pursuer and the target must be equal. Therefore, one may question the applicability of Dubins result in this case, which says the optimal path should be a concatenation of three segments. In the following, we prove, assuming the initial position of the pursuer and any point on the circle is greater than $4\rho$, that the optimal solution to the ITOCP is a Dubins path. Furthermore, we present an iterative algorithm to find the point of interception on the target circle. The iterative algorithm also yields a feasible solution when the $4\rho$ condition does not hold, however, the optimality is not guaranteed.

We summarize the  contributions of this paper as follows:

\begin{itemize}
    \item Under $4\rho$ condition, we prove that the length of the Dubins path $D_{CSC}(\alpha)$ from an initial configuration ($p_i, \theta_i$) to a final point ($\alpha$) on the circle  is a continuous function with respect to $\alpha$.
    \item Under $4\rho$ condition, we prove that the optimal path that intercepts a target moving on a circle at constant speed is a Dubins path, \textit{i.e.} the path contains at most three segments circle-straight line-circle.
    \item We present an iterative algorithm to find the optimal interception point on the target circle and corresponding pursuer's path.
\end{itemize}

\subsection{Solution Approach}
We give a brief explanation of the solution method developed to solve the ITOCP. Let $D_{CSC}(\alpha) \triangleq \min \{D_{LSR}, D_{LSL}, D_{RSL}, D_{RSR} \}$ denote the minimum length (Dubins) path between the initial position of the pursuer and a point at angular position $\alpha$ on the circle. A plot of the lengths of the four types of CSC paths as a function of $\alpha$ is shown in Fig.~\ref{fig:LengthsCSC}. We show that $D_{CSC}(\alpha)$ is a continuous function in the interval $\alpha \in [0, 2 \pi)$, and we know it is non-monotonic from \cite{Manyam2018Shortest}. The function $D_{CSC}(\alpha)$ for any positive value of $\alpha \in [0, \infty)$) is given as the following:
\begin{equation}
   D_{CSC}(\alpha)=D_{CSC} (\mbox{mod}(\alpha,2\pi)). 
\end{equation}
Note, $D_{CSC}(\alpha)$ is periodic, bounded and continuous. This function is useful as it can be used to find the travel time for the pursuer from its initial position, to a point $\alpha$ on the target circle is given by $T_p(\alpha) = D_{CSC}(\alpha)/v_p$. 

The target starts from  an initial position $\alpha_i$ on the circle, and it travels at constant speed. The travel time for the target to any position on the circle is equal to $T_t (\alpha) = (\alpha - \alpha_i)r_t/v_t$. Clearly $T_t (\alpha)$ is linearly increasing function with a minimum of $0$, and $T_p(\alpha)$ is a bounded periodic function. Therefore, the two functions $T_p (\alpha)$ and $T_t (\alpha)$ are bound to intersect for some $\alpha = \alpha^*$. 

The plots of the travel times of the pursuer and the target, $T_p(\alpha)$ and $T_t(\alpha)$ are shown in Fig. \ref{fig:PrTrTimes}. We will prove that the function $D_{CSC}(\alpha)$ is continuous and bounded. In the next section we analyze the four CSC paths $\{LSR, LSL, RSL, RSR \}$, and prove the continuity of the $D_{CSC}(\alpha)$ with respect to the angular position $\alpha$.

\begin{figure}[htpb]
\centering
\includegraphics[width=0.9\columnwidth]{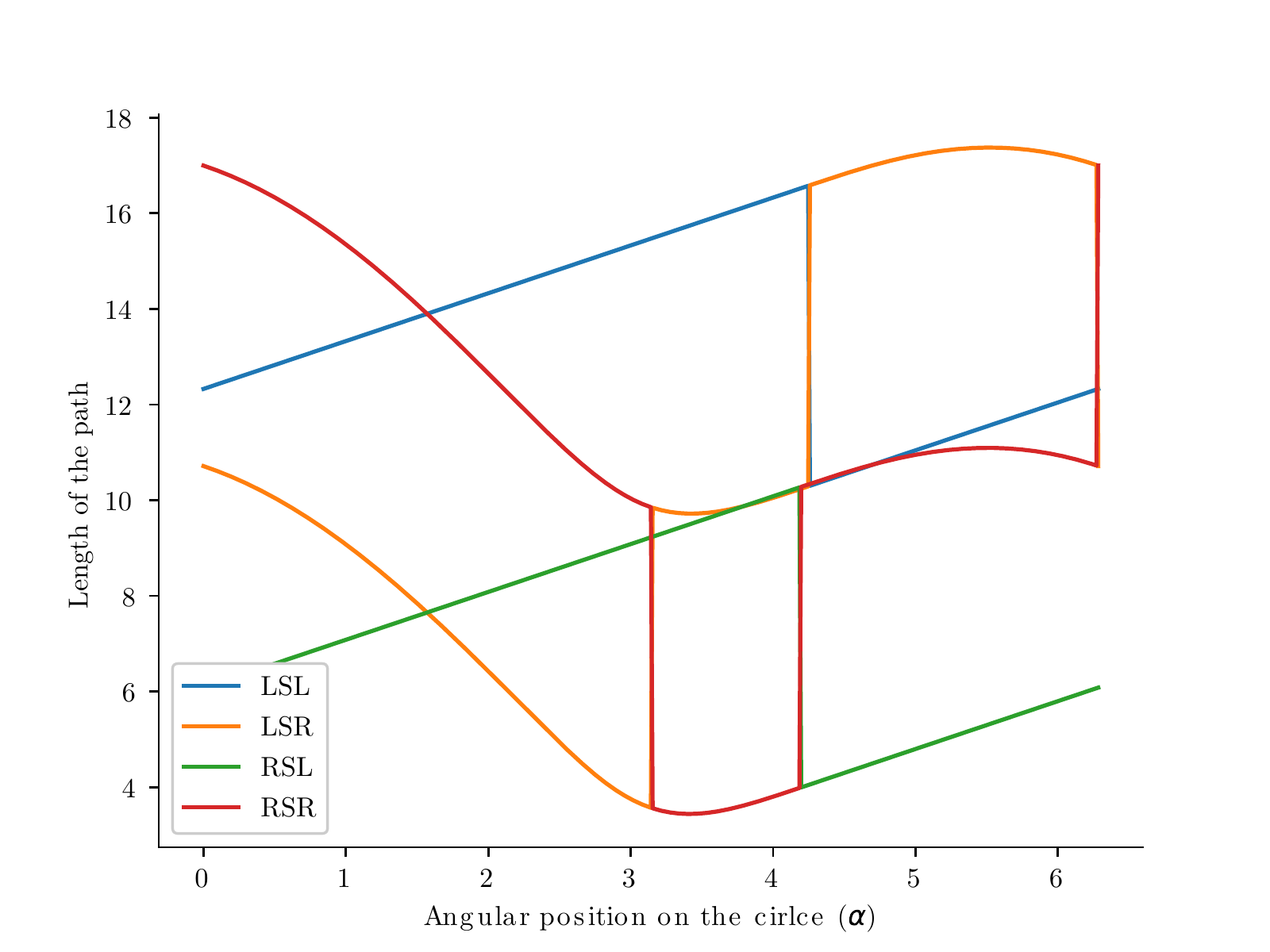}
\caption{Plot of the lengths of the four CSC type Dubins paths versus angular position ($\alpha$)}
\label{fig:LengthsCSC} 
\end{figure}

\begin{figure}[htpb]
\centering
\includegraphics[width=0.9\columnwidth]{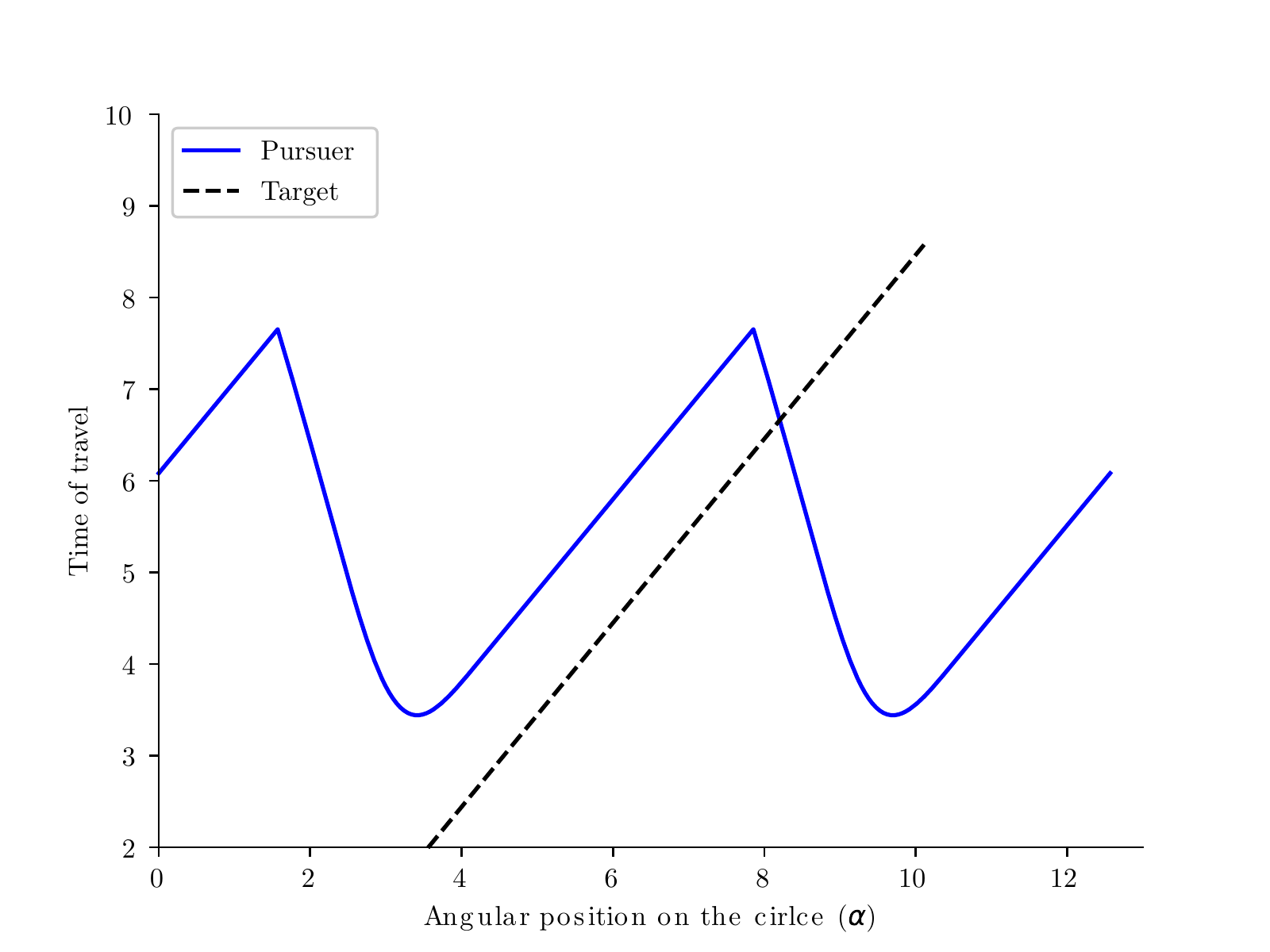}
\caption{Plot of the travel times of the pursuer ($T_p$) and the target ($T_t$)}
\label{fig:PrTrTimes} 
\end{figure}

\section{Analysis of Dubins Paths} \label{sec:dubpaths}

Let $D_{CSC}(\alpha)$ be the length of the shortest CSC path, $D_{CSC} = \min\{D_{LSL}, D_{LSR}, D_{RSR}, D_{RSL} \}$, starting from an the initial position ($p_i$) and heading ($\theta_i$), to a final point defined by the angular position $\alpha$, and the final direction is tangent to the target circle. For all the analysis we present here, we assume the rotational direction at the target circle is clockwise, and the final heading $\theta$ is given as $\theta = \alpha - \frac{\pi}{2}$. The case when the target is traveling counter-clockwise has symmetry between the clockwise case and can be found in the same fashion. We also assume that the minimum turn radius of the pursuer and the radius of target circle are equal. Though the results presented also apply when those two values are not equal, we did not present the proofs due to the page restrictions.

\begin{theorem} \label{thm:csccont}
Under the $4\rho$ condition, the function $D_{CSC}(\alpha)$ is continuous function in its domain $\alpha \in [0, 2 \pi ]$.
\end{theorem}

Proof of Theorem \ref{thm:csccont} will be in the lemmas that follow. 

For any CSC path, we know from~\cite{Manyam2018Shortest} that there may exist a discontinuity if the first circular arc or the final circular arc disappears. For example, consider an $LSL$ path shown in Fig.~\ref{fig:LSLtoLSRb}, starting at $p_i$ and ends on the target circle. When the final point on the target circle is at $\alpha=3.2$, the three segments of the LSL path exist. As the final point is moved in clockwise direction, the length of the third segment reduces, and finally disappears at $\alpha = 2.4$. If the final point is further moved to $\alpha=2$, there is a jump in the length of the final arc, and this would be $2\pi \rho$ at $\alpha=2.4-\epsilon$. This causes the discontinuity in the length of the LSL path at $\alpha=2.4$, shown in Fig.~\ref{fig:LSLtoLSRa}. However, if the mode of the CSC path changes to LSR here, we claim that there would not be this discontinuity, \textit{i.e}, the length $\min \{D_{LSL}, D_{LSR} \}$ is continuous at $\alpha=2.4$. We will prove this more rigorously in the proof of Lemma~\ref{lem:LSLtoLSR}. A similar discontinuity occurs when the first arc disappears for the LSL path at $\alpha=3.35$ shown in Fig.~\ref{fig:LSLtoRSL}. This discontinuity would not occur if the mode of the CSC path changes to RSL, and the length $\min \{D_{LSL}, D_{RSL} \}$ is continuous at $\alpha=3.35$.

\begin{figure}[htpb]
\centering
\subfigure[The length of the paths LSR and RSR versus the angular position on the target circle]{\includegraphics[width=0.8\columnwidth]{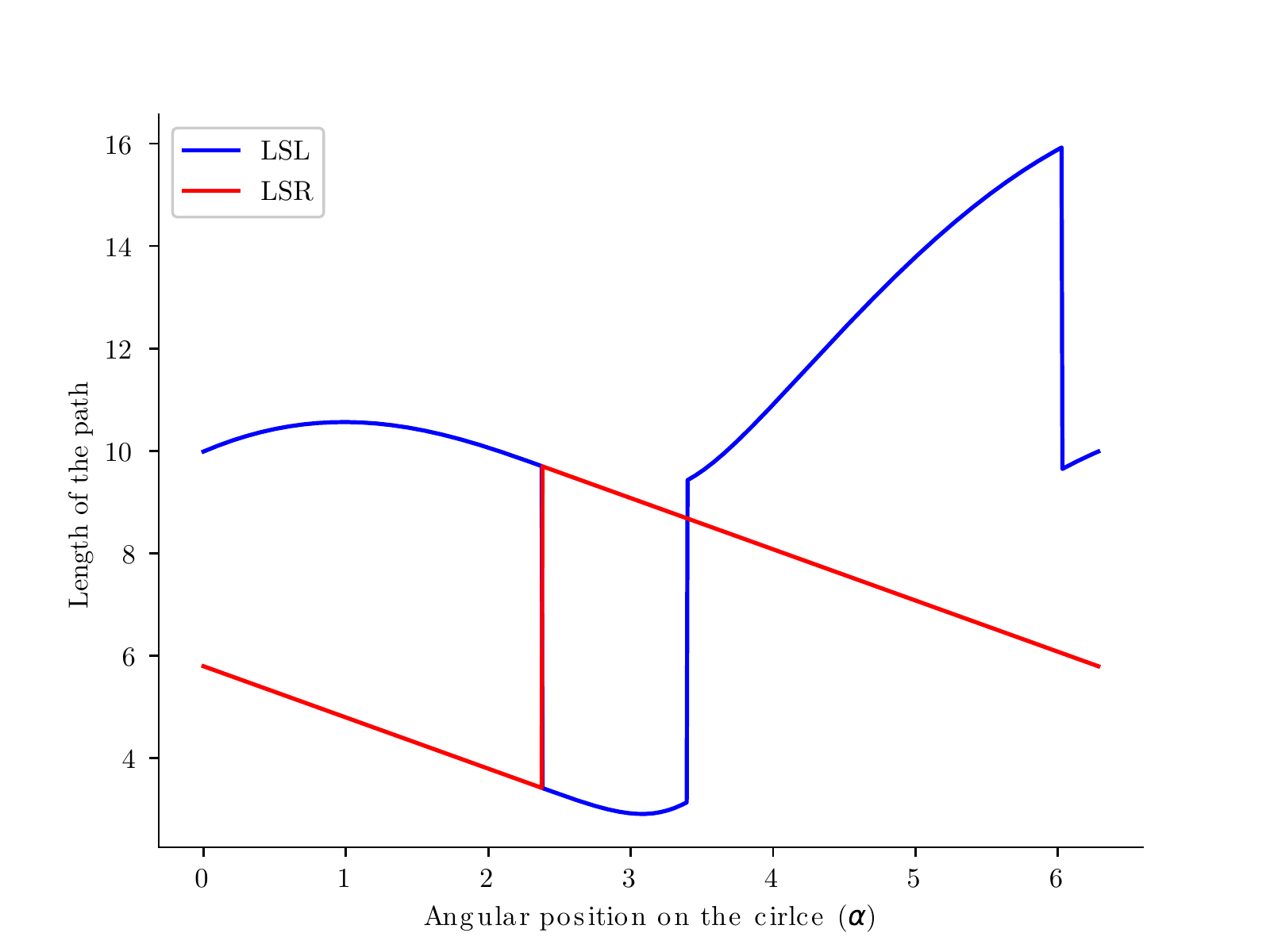} \label{fig:LSLtoLSRa} } 
\subfigure[This figure shows the transitioning of the Dubins path from LSR to LSL as the angular position is moved on the target circle]{\includegraphics[width=0.8\columnwidth]{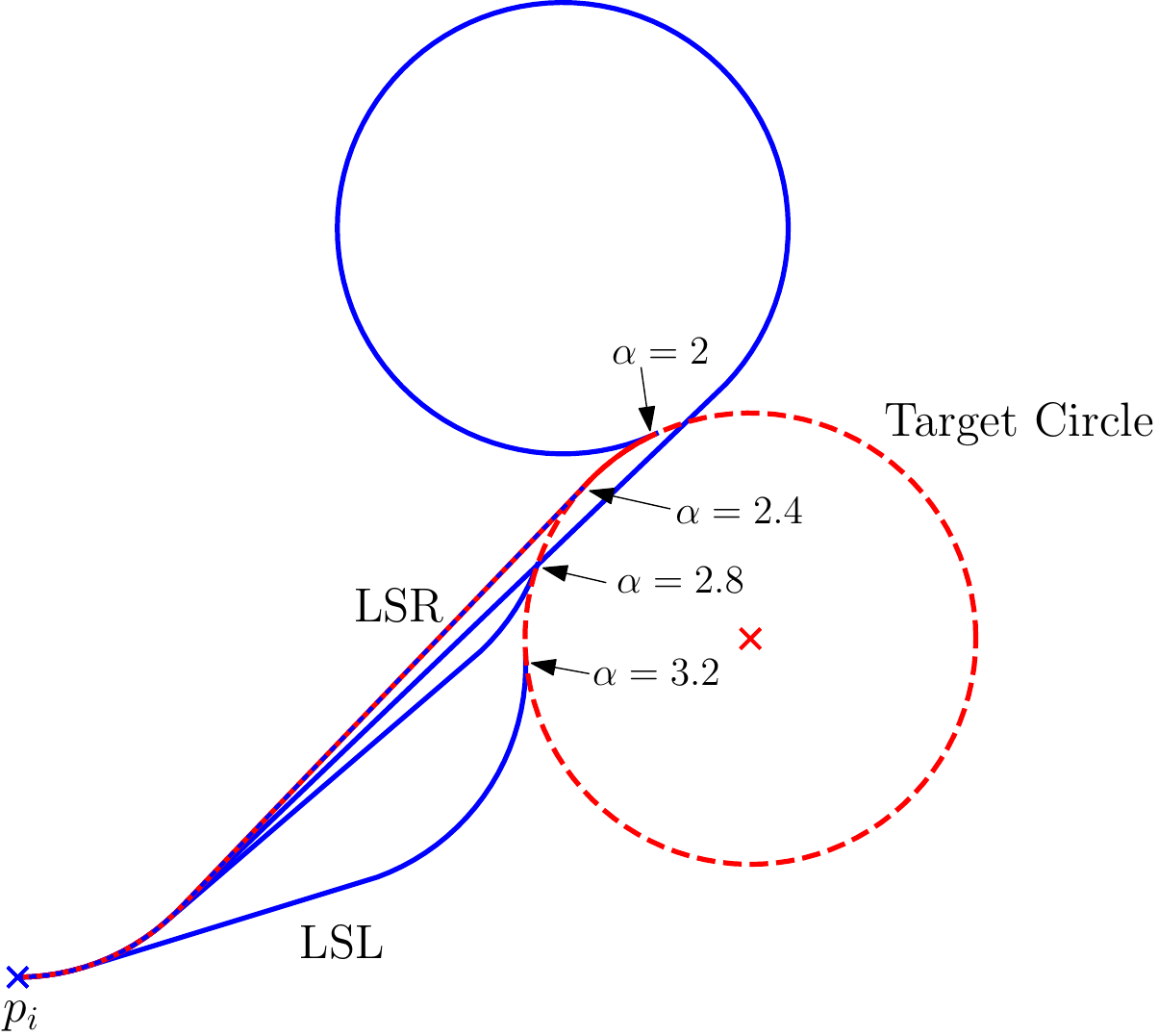}\label{fig:LSLtoLSRb}}
\caption{Transitioning of Dubins path from LSR to LSL}
\label{fig:LSLtoLSR} 
\end{figure}

\begin{figure}[htpb]
\subfigure[The length of the paths LSL and RSL versus the angular position on the target circle]{\includegraphics[width=0.8\columnwidth]{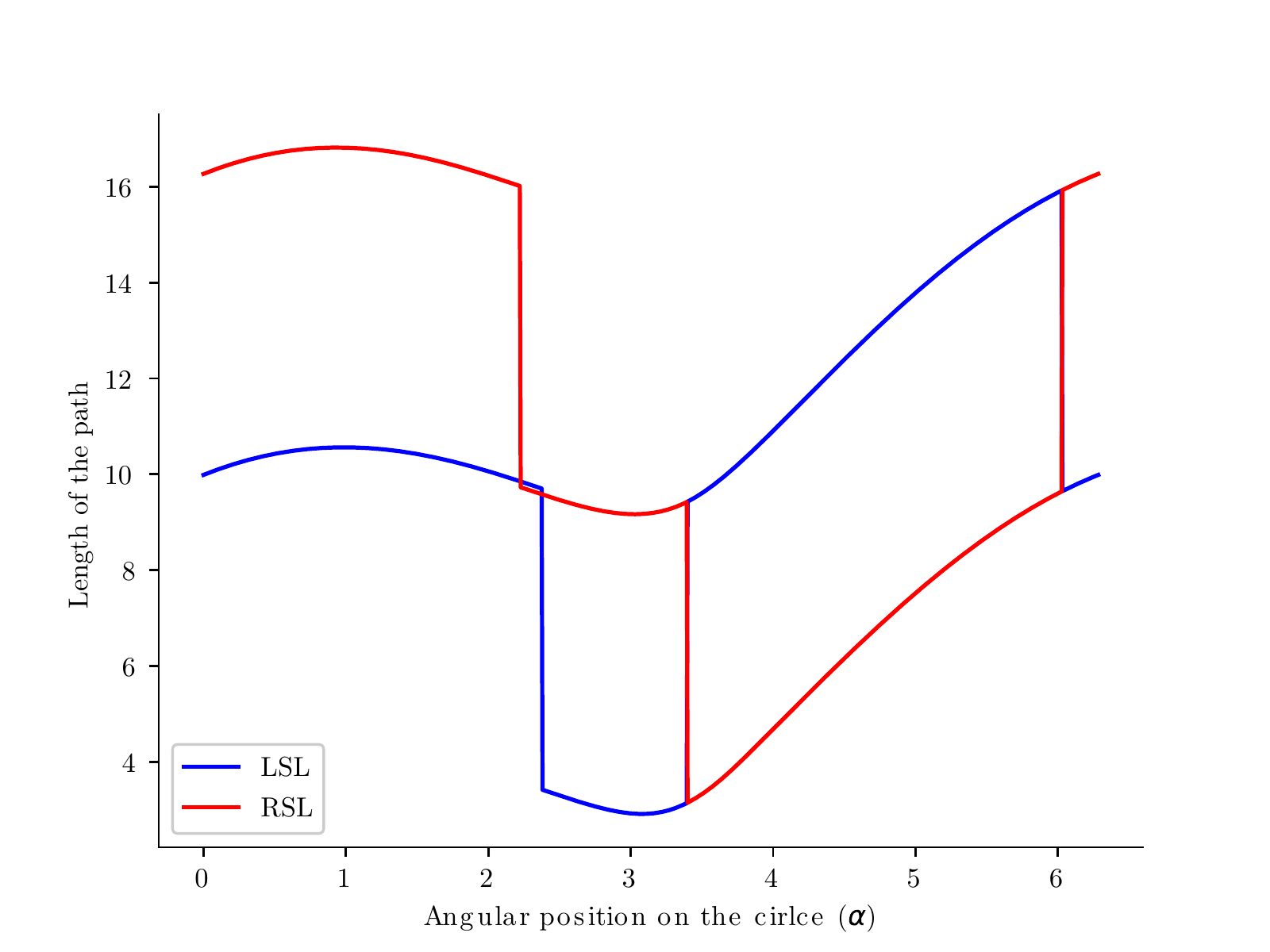}\label{fig:LSLtoRSLa} } 
\subfigure[This figure shows the transitioning of the Dubins path from LSL to RSL as the angular position is moved on the target circle]{\includegraphics[width=0.8\columnwidth]{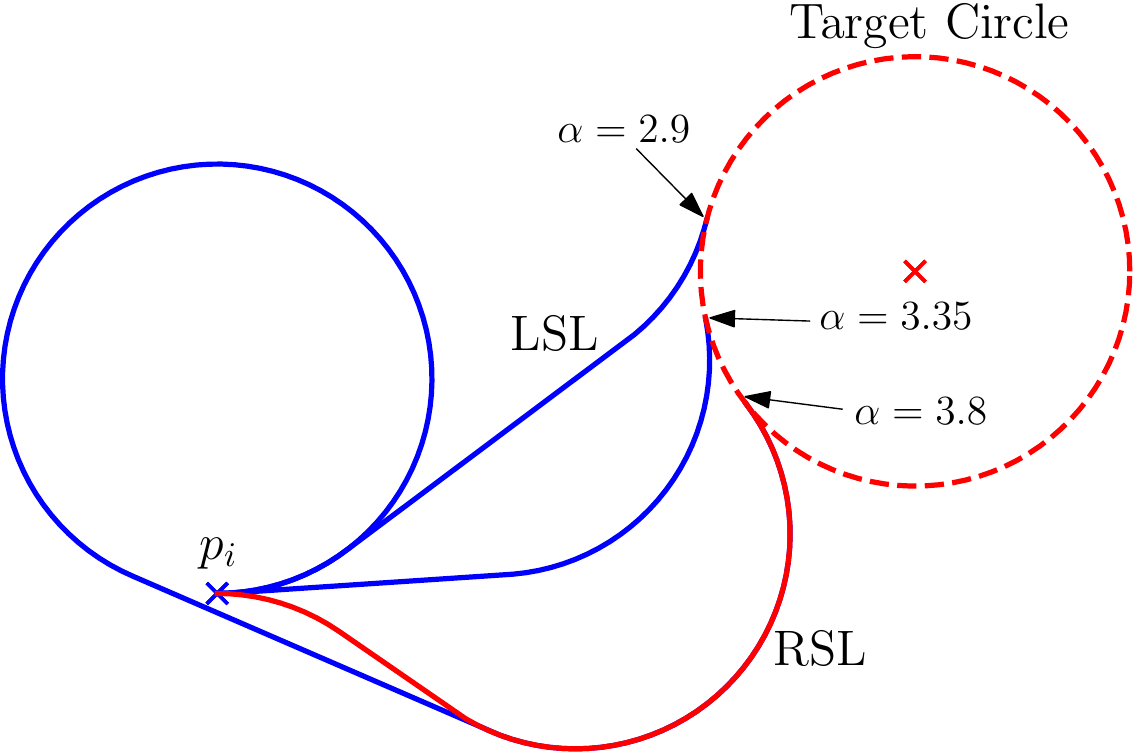} \label{fig:LSLtoRSLb} }
\caption{Transitioning of Dubins path from LSL to RSL}
\label{fig:LSLtoRSL} 
\end{figure}

For any path in the set \{LSL, LSR, RSL, RSR\}, starting from an initial configuration and ending at a final point on the target circle with heading tangential to the circle, and if the final point is moved around the circle, a discontinuity may occur when the first arc or the second arc disappear. We claim that the discontinuity will not occur if the mode of the arc that disappears, changes from L to R or R to L. In a general scenario, either both the first arc and second arc may disappear simultaneously or only one of the arcs may disappear. If only one of the arc disappears, we will prove in the following Lemmas \ref{lem:LSLtoLSR} and \ref{lem:LSLtoRSL} that the length of the CSC path remains continuous if the arc mode changes. For the case when both arcs disappear, the proof follows by combining those two Lemmas.

\subsection{Continuity of CSC paths}
When only one arc disappears, there could be four different change of modes \textit{viz.} $(i)$ LSL - LSR, $(ii)$ RSR - RSL, $(iii)$ LSL - RSL and $(iv)$ RSR - LSR. Among these four, $(ii)$ and $(iv)$ are mirror images of $(i)$ and $(iii)$. We consider the cases $(i)$ and $(iii)$, and prove the continuity of the min length path. Clearly, the other two cases follow due to symmetry.

Let the position on the target circle at which the second arc of the LSR and LSL paths disappear be $\alpha_{LS}$. The length of these two paths has a discontinuity at $\alpha = \alpha_{LS}$. Similarly, let $\alpha_{SL}$ be the position on the target circle at which the first arc of LSL and RSL paths disappear. 

\begin{lemma} \label{lem:LSLtoLSR}
The minimum length of the two paths $\min \{D_{LSR}, D_{LSL} \}$ is continuous at $\alpha = \alpha_{LS}$.
\end{lemma}
\begin{proof}
The length of the LSL path $D_{LSL}(\alpha)$ is sum of the three segments, straight line $L^{LSL}_S(\alpha)$, first and second arcs $\rho\phi^{LSL}_1(\alpha)$ and $\rho\phi^{LSL}_2(\alpha)$. And let $L^{LSR}_S$, $\phi^{LSR}_1$ and $\phi^{LSR}_2$ be the corresponding lengths of the LSR path.\footnote{For brevity, we represent these functions as $L^{LSL}_S$, $\phi^{LSL}_1$ and $\phi^{LSL}_2$.} These lengths are given as the following:

\begin{align}
    L^{LSL}_S &= \sqrt{(c_x + 2\rho \cos{\alpha})^2 + (c_y+2\rho\sin{\alpha}-\rho )^2 },\\
    \phi^{LSL}_1 &= \mod \left( \arctan \left( \frac{c_y+2\rho\sin{\alpha}-\rho}{c_x + 2\rho\cos{\alpha}} \right), 2\pi \right), \label{eq:phi1LSL} \\
    \phi^{LSL}_2 &= \mod \left(\alpha-\phi^{LSL}_1 - \frac{\pi}{2}, 2 \pi \right), \\
    L^{LSR}_S &= \sqrt{c_x^2 + (c_y-\rho)^2 - 4\rho^2}, \\
    \phi^{LSR}_1 &= \mod \left(\psi_1 + \psi_2, 2\pi \right), \\
    \phi^{LSR}_2 &= \mod \left(\phi^{LSR}_1 -\alpha +\frac{\pi}{2}, 2\pi \right) ,
\end{align}
where $\psi_1$ and $\psi_2$ are given by
\begin{align*}
\psi_1 &= \arctan\left(\frac{2\rho}{L^{LSR}_S} \right), \\
\psi_2 &= \arctan\left( \frac{c_y-\rho}{c_x} \right).
\end{align*}

At $\alpha = \alpha_{LS}$, the second arc disappears for the LSL and LSR paths, and these two paths degenerate to LS paths. Clearly, there could be only one possible LS path with the final heading as clockwise tangent to the target circle, and therefore $D_{LSL}(\alpha_{LS}) = D_{LSR}(\alpha_{LS})$. 

Now we prove that the in the neighborhood of the $\alpha_{LS}$, length of the paths $D_{LSL}(\alpha)$ is right continuous and $D_{LSR}(\alpha)$ is left continuous.

The function $L^{LSL}_S (\alpha)$ is continuous everywhere, $\phi^{LSL}_1 (\alpha)$ is decreasing at $\alpha_{LS}$, and due to the assumption that only one arc disappears, $\phi^{LSL}_1 (\alpha)$  is continuous at $\alpha = \alpha_{LS}$. For $\delta >0$, we have $$ \phi^{LSL}_1 (\alpha_{LS} + \delta) = \phi^{LSL}_1 (\alpha_{LS}) - k \delta.$$ Therefore, $\phi^{LSL}_2 (\alpha_{LS} + \delta)$ can be given by
\begin{align*}
    \phi_2^{LSL}(\alpha_{LS}+\delta) &= \mod \left( \alpha + \delta -\phi_1^{LSL}(\alpha_{LS}) + k\delta - \frac{\pi}{2}, 2\pi \right)\\
    &= \mod ( k\delta+\delta,2\pi) \\
    &= k\delta+\delta.
\end{align*}
Therefore, for any $\epsilon  \ni \phi_2^{LSL}(\alpha_{LS}+\delta)-\phi_2^{LSL}(\alpha_{LS}) <\epsilon$, we can find $\delta$. This proves the right continuity of $\phi_2^{LSL}(\alpha)$, and thus $D_{LSL}(\alpha)$ is right continuous at $\alpha=\alpha_{LS}$.

In the LSR path, the first arc $\phi^{LSR}_1$, and the straight line $L^{LSR}_S$ are constants.
\begin{align*}
    & \phi_2^{LSR}(\alpha_{LS}-\delta) - \phi_2^{LSR}(\alpha_{LS}) \\
    &= \mod \left( \phi_1^{LSR}(\alpha_{LS}-\delta) - \phi_1^{LSR}(\alpha_{LS})  + \delta , 2\pi\right) \\
    &= \delta.
\end{align*}
Clearly, $\phi_2^{LSR}(\alpha)$ is left continuous, and thus the length $D_{LSR}(\alpha)$ is left continuous at $\alpha=\alpha_{LS}$.
It is straight forward to see that $D_{LSR}(\alpha) < D_{LSL}(\alpha)$ when $\alpha < \alpha_{LS}$, $D_{LSR}(\alpha) > D_{LSL}(\alpha)$ when $\alpha > \alpha_{LS}$, and $D_{LSR}(\alpha) = D_{LSL}(\alpha)$ when $\alpha = \alpha_{LS}$. Therefore, $\min \{D_{LSR}(\alpha), D_{LSL}(\alpha) \}$ is continuous in the neighborhood of $\alpha_{LS}$.

\end{proof}

\begin{lemma} \label{lem:LSLtoRSL}
The minimum of the length of the two paths $\min \{D_{LSL}(\alpha), D_{RSL}(\alpha) \}$ is continuous at $\alpha = \alpha_{SL}$.
\end{lemma}
\begin{proof}
We present the proof of this lemma using geometric perturbations. Also, an analytical proof similar to the proof of Lemma~\ref{lem:LSLtoLSR} is presented in the Appendix.

\begin{figure}
    \centering
    \subfigure[SL to LSL]{\includegraphics[width=0.6\columnwidth]{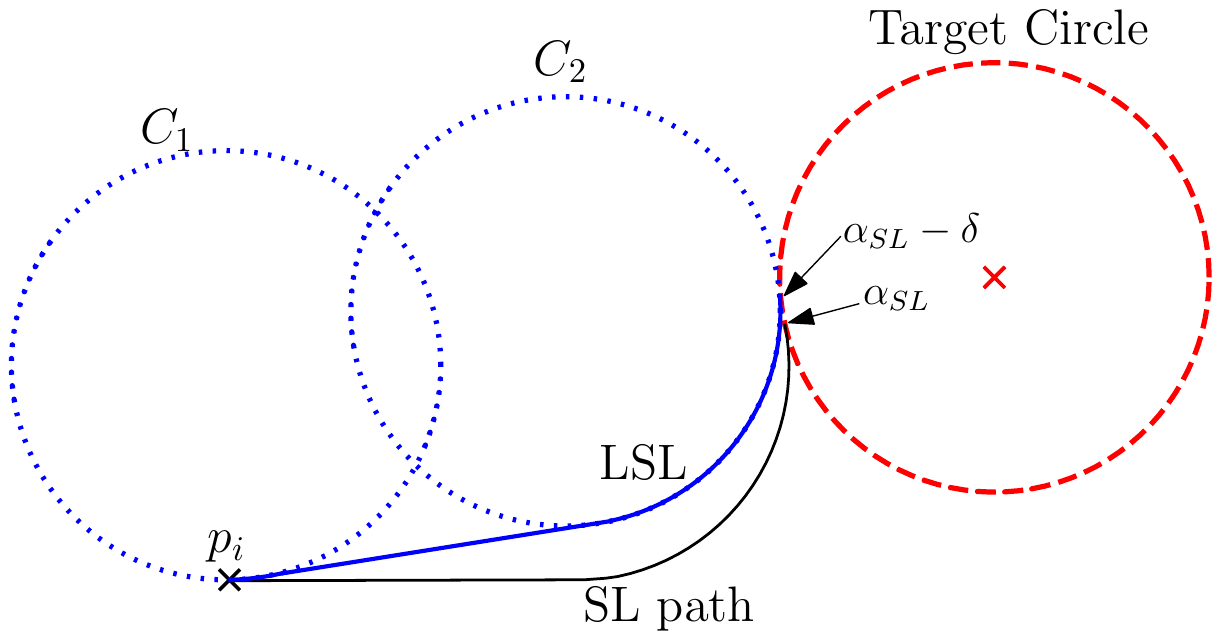} \label{fig:lem2a}}
    \subfigure[SL to RSL]{\includegraphics[width=0.6\columnwidth]{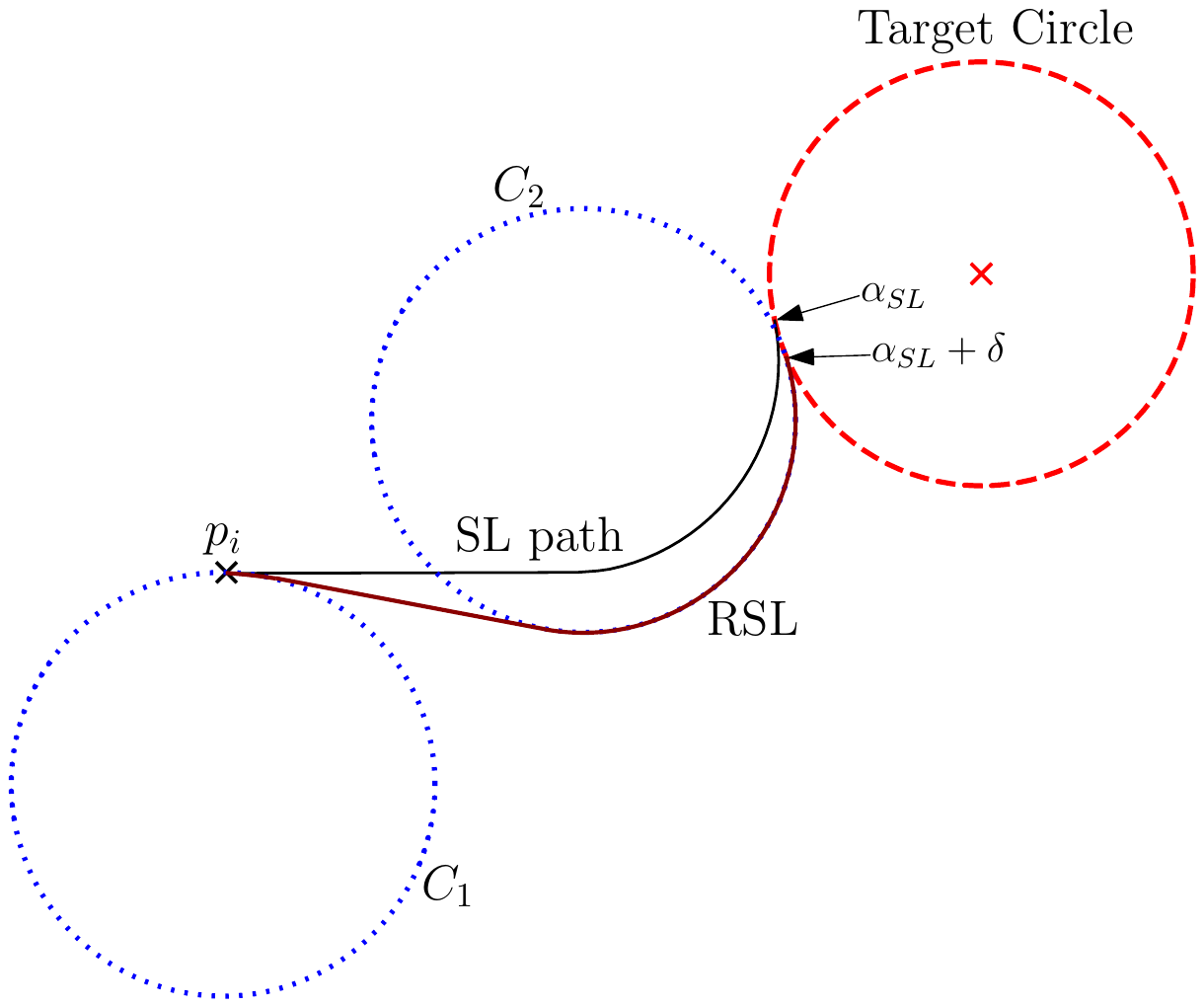}\label{fig:lem2b}}
    \caption{Perturbation of the degenerate SL path around $\alpha_{SL}$}
    \label{fig:lem2}
\end{figure}
It is straightforward to see that when the LSL and RSL paths degenerate to SL path, the lengths are equal. The position $\alpha_{SL}$, at which the two paths degenerate to SL is illustrated in Figs.~\ref{fig:lem2a} and~\ref{fig:lem2b}. The blue dotted circles $C_1$ and $C_2$ represent the first and second complete circles of the CSC paths.
If the final point is perturbed to $\alpha_{SL}-\delta$ on the target circle, we could perturb the second circle to be tangent to the target circle at $\alpha_{SL}-\delta$, and construct the LSL path as shown in Fig.~\ref{fig:lem2a}. A similar RSL path could be constructed as shown in Fig.~\ref{fig:lem2b} for a final position at $\alpha_{SL}+\delta$. Clearly, for infinitesimal $\delta$, the change in the length of the three segments in either direction is infinitesimal. Therefore, at $\alpha=\alpha_{SL}$, $D_{LSL}(\alpha)$ is left continuous, and $D_{RSL}(\alpha)$ is right continuous, and the minimum of these two lengths $\min \{D_{LSL}, D_{RSL}\}$ is continuous.
\end{proof}
The proofs for the Lemmas \ref{lem:LSLtoLSR} and \ref{lem:LSLtoRSL} completes the proof of the Theorem \ref{thm:csccont}.

\section{Optimal Interception Path} \label{sec:optpath}

The time of travel $T_p(\alpha)$ for the pursuer to reach a point at $\alpha$ on target circle is given by $D_{CSC}(\alpha)/v_p$, and it is continuous for $\alpha \in [0, \infty)$. Due to the $4\rho$ condition, the minimum of $T_p(\alpha)$ is strictly greater than $0$. As the target is traveling at constant speed, the time of travel $T_t(\alpha)$ for the target to reach the position $\alpha$  on the target circle is continuous. Since $T_p(\alpha)$ is bounded and $T_t(\alpha)$ is linearly increasing, there always exist $\alpha_l$ and $\alpha_u$ such that
\begin{equation}
    T_t(\alpha_l) < T_p(\alpha_l), \quad T_t(\alpha_u) > T_p(\alpha_u).
    \label{limitsfeas}
\end{equation}
Therefore, by intermediate value theorem, $T_p(\alpha) = T_t(\alpha)$ for some $\alpha^* \in (\alpha_l, \alpha_u)$. We could find $\alpha^*$ using a simple bisection algorithm. To find the solution efficiently, it is necessary to have a good limits of the search domain $\alpha_l$ and $\alpha_u$. Here, we give a systematic procedure to set the limits for $\alpha$. From \cite{Manyam2018Shortest}, we know where the minimum of the Dubins CSC path occurs, and let $\alpha_{\min} = \arg \min_{\alpha \in [0, 2\pi)} D_{CSC}(\alpha)$. Note that the minimum of this function occurs periodically at $\alpha_{\min}, \alpha_{\min}+2\pi,\ldots$. 

Initially, we set the lower limit to the initial location of the target, $\alpha_l= \alpha_i$, and the upper limit to the first position greater than $\alpha_i$ where the minimum occurs, $\alpha_u = \min_{\alpha > \alpha_i} D_{CSC}(\alpha)$. We check if the conditions in (\ref{limitsfeas}) are satisfied, and if not, we update the limits as $\alpha_l := \alpha_u$ and $\alpha_u := \alpha_u+2\pi$. We present the pseudo code of the bisection algorithm in Algorithm~\ref{alg:bisect}.

\begin{algorithm}
    \setstretch{1.1}
  \caption{Pseudo-code of the bisection algorithm}
  \label{alg:bisect}
  \begin{algorithmic}[1]
    \vspace{1ex}
    \Input $p_i$, $\theta_i$, $C_t$, $\alpha_i$, $\rho$ and $r_t$.
    \Output Angular position of the interception point on the target circle, $\alpha^*$
    \State $\alpha_l \gets \alpha_i$, $\alpha_u = \min_{\alpha > \alpha_i} D_{CSC}(\alpha)$
    \While{ $(T_t(\alpha_l) > T_p(\alpha_l))$ || $(T_t(\alpha_u) < T_p(\alpha_u))$}
    \State $\alpha_l \gets \alpha_u$
    \State $\alpha_u \gets \alpha_u+2\pi$
    \EndWhile
    
    \State $\delta t \gets \infty $ 
    \While{$\delta t > \epsilon$}
    \State $\alpha^* \gets (\alpha_l + \alpha_u)/2$    
    \State $\delta t \gets T_p(\alpha^*) - T_t(\alpha^*) $
    
    \If{$\delta t > 0$}
    \State $\alpha_l \gets \alpha^*$
    \Else
    \State $\alpha_u \gets \alpha^*$
    \EndIf
    \EndWhile
    
  \end{algorithmic}
\end{algorithm}

\begin{prop}
The interception point found by Algorithm~\ref{alg:bisect} is optimal, and therefore the optimal interception path is a Dubins path.
\end{prop}
\begin{proof}
We will prove this proposition using contradiction. Let us assume there exists an optimal non-Dubins path for the pursuer, that intercepts the target at $\hat{\alpha}$; let its distance be $\hat{D}(\hat{\alpha})$, and the corresponding travel time be $\hat{T}_p(\hat{\alpha})$. By the definition of Dubins path, $D_{CSC}(\hat{\alpha}) < \hat{D}(\hat{\alpha})$. Also the travel time $T_p(\hat{\alpha})$ of the pursuer along the Dubins path is less than the assumed path, \textit{i.e.} $T_p(\hat{\alpha}) < \hat{T}_p(\hat{\alpha})$, and this implies $T_p(\hat{\alpha}) < T_t(\hat{\alpha})$. There always exist an $\alpha_l$ such that $T_t(\alpha_l) < T_p(\alpha_l)$. By intermediate value theorem $T_t(\alpha^*) = T_p(\alpha^*)$ for some $\alpha^* \in (\alpha_l, \hat{\alpha} )$. 
This implies that this new path intercepts the target at a position $\alpha^*$, and $\alpha^* < \hat{\alpha}$.  Therefore, the assumption is incorrect, and this completes the proof.
\end{proof}

We test the algorithm to find the optimal solution to the ITOCP using a simulated scenario. The initial position of the pursuer is at the origin, and its initial heading is $0$ degrees with respect to positive $x$-axis. The center of target circle is at $(-4, 3)$, and the initial target's location is at $\alpha_i = \pi$. The pursuer and evader's speeds are $1$ m/sec and $1.2$ m/sec. The solution to the ITOCP is computed by the Algorithm~\ref{alg:bisect}, and the results are shown in Fig.~\ref{fig:results}. The travel times for the pursuer and the target are shown in Fig.~\ref{fig:resultsa} and the plot of the path of the pursuer between $p_i$ and the interception point is shown in Fig.~\ref{fig:resultsb}.

\begin{figure}
    \centering
    \subfigure[Travel times of the target and the pursuer]{\includegraphics[width=0.75\columnwidth]{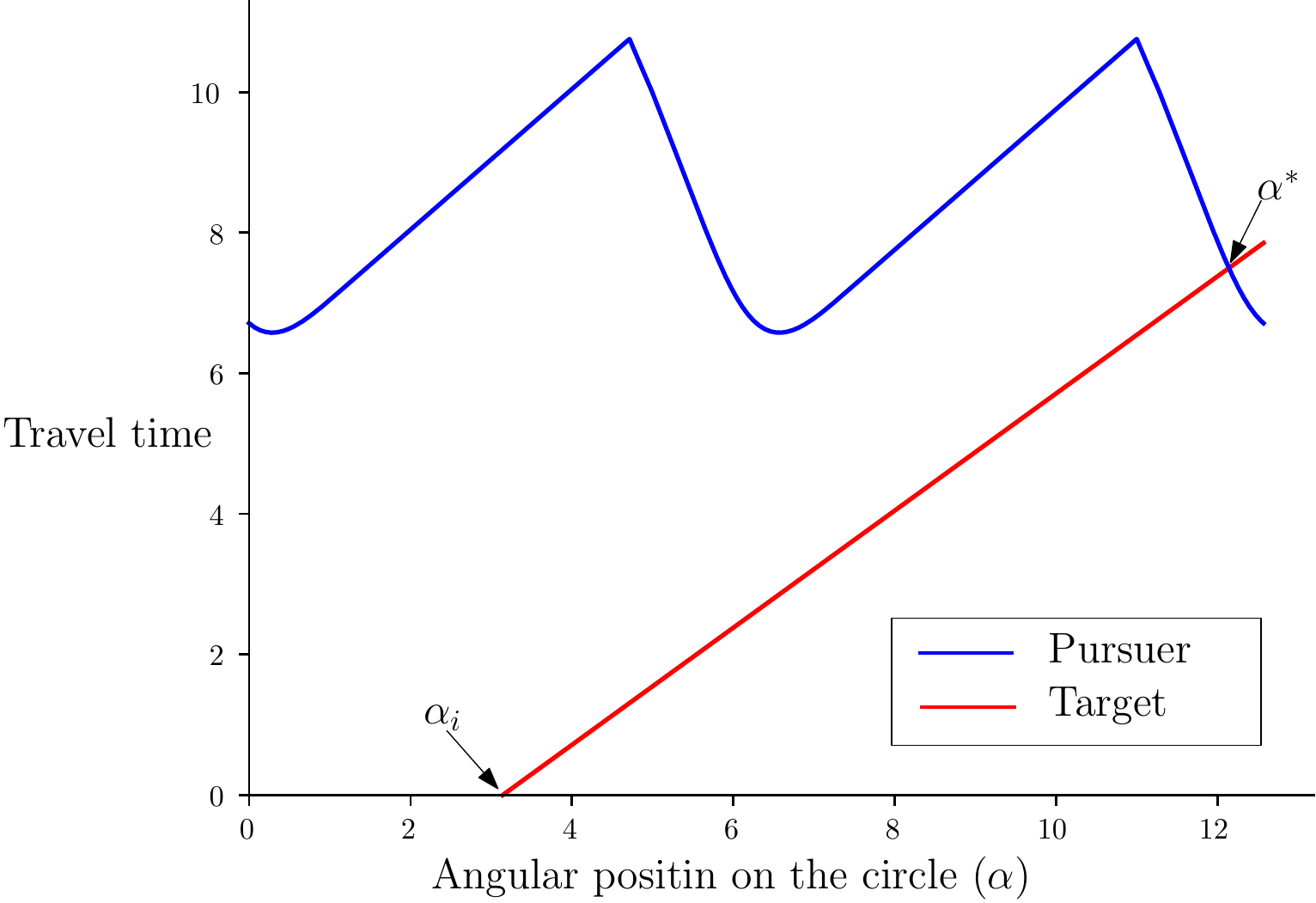} \label{fig:resultsa}}
    \subfigure[Path of the pursuer intercepting the target]{\includegraphics[width=0.75\columnwidth]{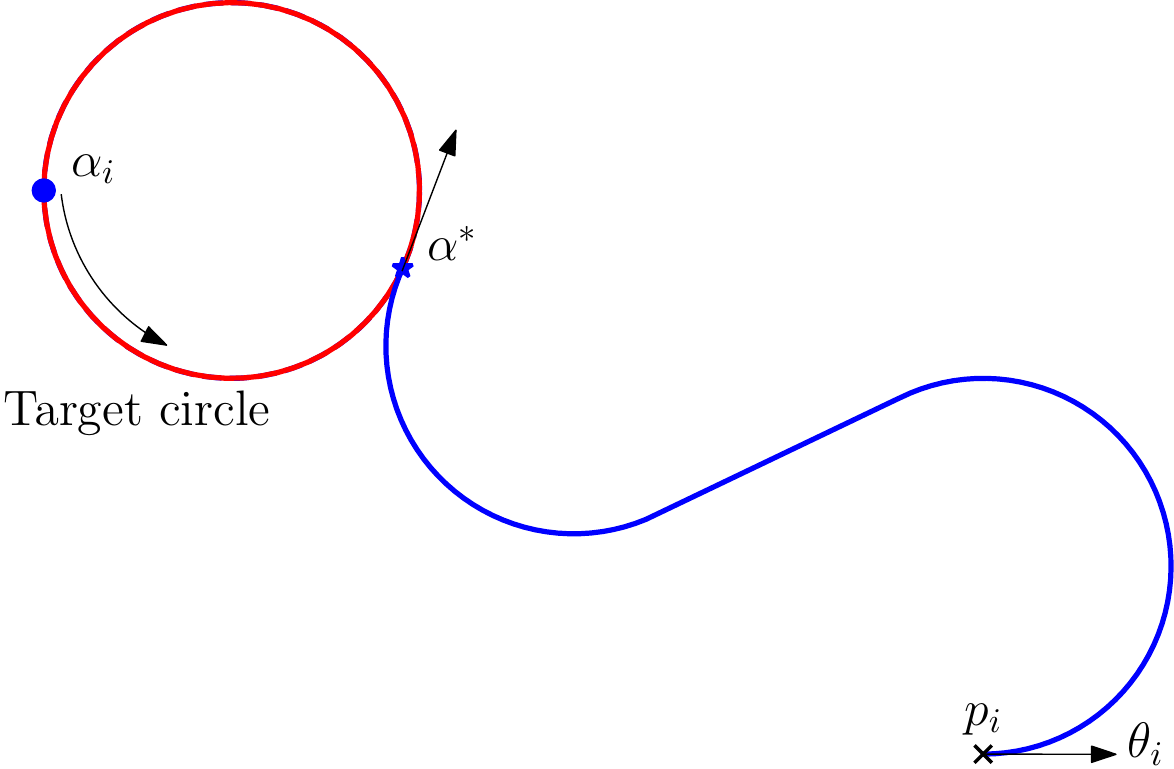}\label{fig:resultsb}}
    \caption{An interception path for a test scenario}
    \label{fig:results}
\end{figure}

\section{CONCLUSIONS} \label{sec:concl}
We proposed a target intercept problem where the pursuer's path must satisfy the minimum turn radius constraints, and the target is traveling on a circle. Under the $4\rho$ condition, we proved that the Dubins CSC path is a continuous and bounded with respect to the final angular position on the target circle. Also we proved that the optimal path to intercept is a Dubins path. We presented an iterative algorithm to find the interception point on the target circle and the corresponding path. The algorithm presented was evaluated using an example test scenario.

\bibliographystyle{ieeetr}
\bibliography{CircleIntercept}

\section*{Appendix} \label{sec:app}
Here we present the analytical proof of Lemma \ref{lem:LSLtoRSL}.
\begin{proof}
It is straight forward that the length of the LSL and RSL paths are equal when they degenerate to SL. Let the final position on the target circle be $\alpha_{SL}$ when these paths degenerate to SL. The length of the first arc of LSL path, $\phi_1^{LSL}$ is given in eq. (\ref{eq:phi1LSL}). Let us define the term inside the modulus function in eq. (\ref{eq:phi1LSL}) as $f(\alpha) := \arctan \left( \frac{c_y+2\rho \sin \alpha-\rho}{c_x+2\rho \cos \alpha} \right)$, and when the first arc of the LSL path disappears, $f(\alpha_{SL}) = 0$. 
To prove that $\phi_1^{LSL}$ is left continuous at $\alpha_{SL}$, it is sufficient to show that $f$ is decreasing function at $\alpha=\alpha_{SL}$. We will show that the first derivative of $f$ with respect to $\alpha$ is negative at $\alpha=\alpha_{SL}$.
\begin{figure}[htpb]
\centering
\includegraphics[width=.75\columnwidth]{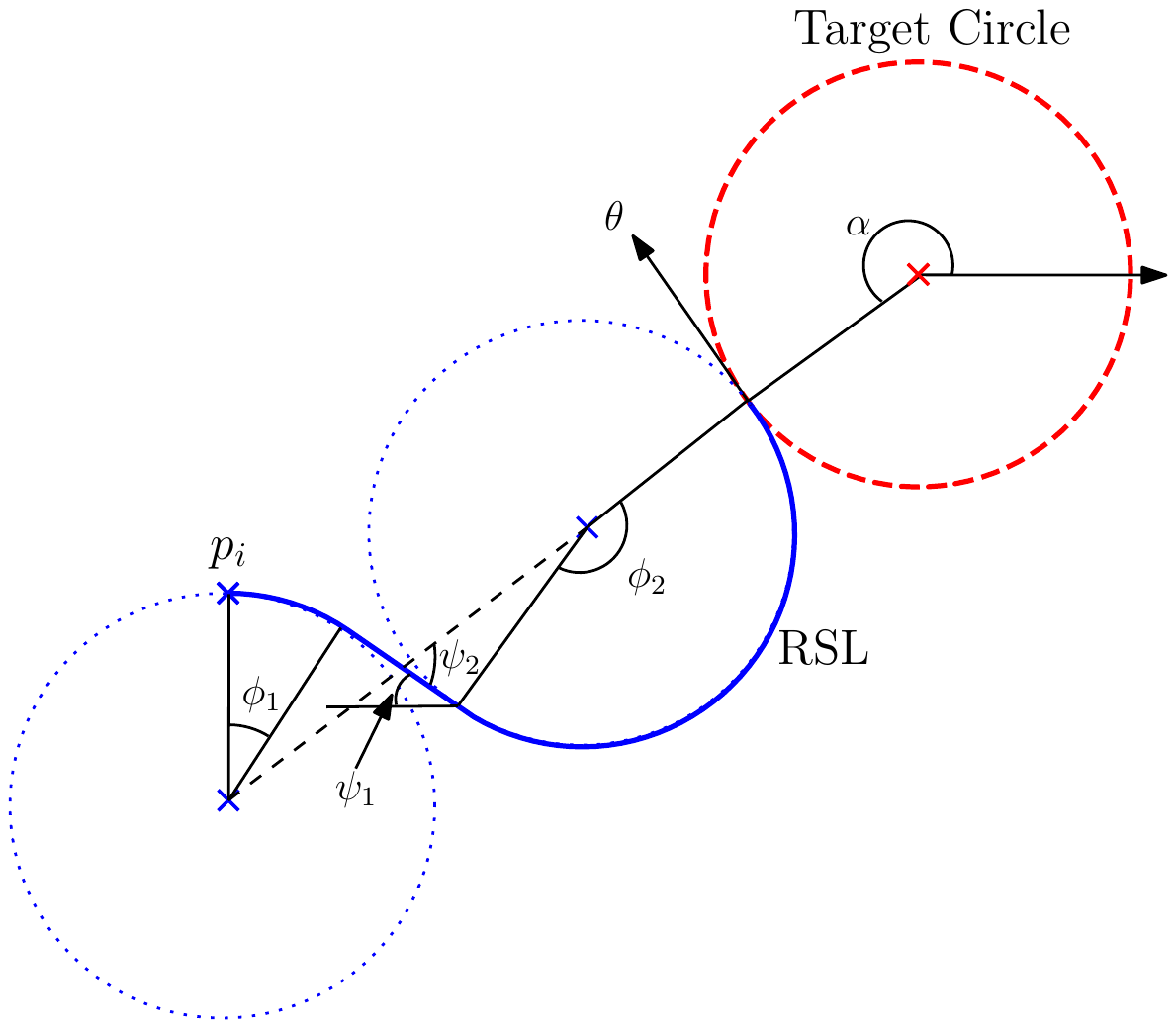}
\caption{RSL path}
\label{fig:lem2rsl} 
\end{figure}
\begin{flalign*}
 &f'(\alpha)  \\
 &\,\, =\frac{(c_x+2\rho \cos \alpha)(2\rho cos \alpha) + (c_y+2\rho \sin \alpha -\rho)(2\rho \sin \alpha)}{(c_x+2\rho \cos \alpha)^2 + (c_y+2\rho \sin \alpha -\rho)^2},   \\
 &\,\, = \frac{2\rho}{L_S^{LSL}}\left( \cos \alpha \cos \phi_1 + \sin \alpha \sin \phi_1 \right),\\
 &\,\, = -\frac{2\rho}{L_S^{LSL}}\sin (\theta-\phi_1), \\
 &\,\, f'(\alpha_{SL}) = -\frac{2\rho}{L_S^{LSL}}\sin (\theta).
\end{flalign*}
When $\phi_1^{LSL}=0$, it is clear from the Fig. \ref{fig:lem2a} that the final heading $\theta$ always lies between $0$ and $\pi$. Therefore, $f(\alpha)$ is decreasing at $\alpha=\alpha_{SL}$, and $\phi_1^{LSL}(\alpha)$ is left continuous at $\alpha_{SL}$.
    
We will prove the right continuity of $\phi_1^{RSL}$ using a similar approach. The first arc of the RSL path is given as (see Fig. \ref{fig:lem2rsl}) $$\phi_1^{RSL} = \mod \left( -\psi_1+\psi_2, 2\pi \right),$$ where $\psi_1$ and $\psi_2$ are given as 
\begin{align*}
    \psi_1 &= \arctan \left( \frac{c_y + 2\rho \sin \alpha + \rho}{c_x + 2\rho \cos \alpha} \right), \\
    \psi_2 &= \arcsin \left(\frac{2\rho}{L_{cc}^{RSL}} \right), \\ 
    L_{cc}^{RSL}&=\sqrt{ (c_x + 2\rho \cos \alpha)^2 + (c_y + 2\rho \sin \alpha + \rho)^2 }.
\end{align*}
The first derivatives of $\psi_1$ and $\psi_2$ are given as the following:
\begin{flalign*}
\psi_1' &= \frac{2\rho}{L_{cc}^{RSL}} \left(\cos \psi_1 \cos \alpha + \sin \psi_1 \sin \alpha \right), \\
&=\sin \psi_2 \cos (\psi_1-\alpha), \\
\psi_2' &= -\tan \psi_2 \sin \psi_2 \sin (\psi_1 - \alpha).
\end{flalign*}
Let $g(\alpha) = -\psi_1 + \psi_2$, and its first derivative is given as follows:
\begin{flalign*}
g'(\alpha) &= -\tan \psi_2 \cos (\alpha - (\psi_1 - \psi_2)), \\
&= \tan \psi_2 \sin (\theta - (\psi_1 - \psi_2)), \\
&= \tan \psi_2 \sin \phi_2.
\end{flalign*}
At $\alpha = \alpha_{SL}$, $\psi_2$ lies between $0$ and $\frac{\pi}{2}$, and $\phi_2$ is greater than $0$ and less than $\pi$. Therefore, $g'(\alpha) >0$ at $\alpha=\alpha_{SL}$, and hence the $\phi_1^{RSL}$ is right continuous. Therefore, $\min \{D_{RSL}(\alpha), D_{LSL}(\alpha)\}$ is continuous at $\alpha=\alpha_{SL}$.
\end{proof}

\end{document}